\documentclass[twoside]{svproc}

\usepackage{pdfx}
\usepackage{tikz}
\usepackage{arydshln}
\usepackage{comment}
\usepackage{stmaryrd}
\usepackage{mathtools}
\usepackage{amsmath}
\usetikzlibrary{shapes}
\usepackage{graphicx}
\usepackage{url}

\def\morder{{M_\text{O}}}

\def\mom{{\mathbb M}}

\def\reals{\mathbb R}
\def\dv{\Delta v}
\newcommand{\bunderline}[1]{\underline{#1}}
\renewcommand{\vec}[1]{{\bunderline{#1}}}

\newcommand{\mat}[1]{{\bunderline{\bunderline{#1}}}}
\def\mom{{\mathbb M}}
\def\dt{{\Delta t}}

\usepackage{lipsum}
\usepackage{amsfonts}
\usepackage{amssymb}
\usepackage{graphicx}
\usepackage{epstopdf}
\usepackage{algorithmic}

\begin{document}

\title{A novel splitting method for Vlasov-Amp\`ere}
\titlerunning{A novel splitting method for Vlasov-Amp\`ere}

\author{James A. Rossmanith$^{1}$ \and
        Christine Vaughan$^{2}$}
\authorrunning{J.A. Rossmanith and C. Vaughan}

\institute{Iowa State University, Department of
          Mathematics, 411 Morrill Road, Ames, Iowa 50011, USA \\
          \email{\href{mailto:rossmani@iastate.edu}{rossmani@iastate.edu}}
          \and
          Physicians Health Plan, 1400 East Michigan Ave, Lansing, Michigan 48912, USA \\
          \email{\href{mailto:cvaughan880@gmail.com}{cvaughan880@gmail.com}}}

\maketitle

\begin{abstract}
Vlasov equations model the dynamics of plasma in the collisionless regime. A standard approach for numerically solving the Vlasov equation is to operator split the spatial and velocity derivative terms, allowing simpler time-stepping schemes to be applied to each piece separately (known as the Cheng-Knorr method). One disadvantage of such an operator split method is that the order of accuracy of fluid moments (e.g., mass, momentum, and energy) is restricted by the order of the operator splitting (second-order accuracy in the Cheng-Knorr case). In this work, we develop a novel approach that first represents the particle density function on a velocity mesh with a local fluid approximation in each discrete velocity band and then introduces an operator splitting that splits the inter-velocity band coupling terms from the dynamics within the discrete velocity band. The advantage is that the inter-velocity band coupling terms are only needed to achieve consistency of the full distribution functions, but the local fluid models within each band are sufficient to achieve high-order accuracy on global moments such as mass, momentum, and energy. The resulting scheme is verified on several standard Vlasov-Poisson test cases.
\end{abstract}








\section{Introduction}
\label{sec:intro}

Plasma is a highly energetic state of matter where electrons have dissociated from their nuclei, creating
a mixture of charged particles that interact through electromagnetic forces and particle-particle collisions. 
In the collisionless approximation, the plasma can be modeled via the celebrated Vlasov equation. In particular, we
are concerned here with the (non-dimensionalized) 
single-species (i.e., streaming electrons with stationary neutralizing background ions)
Vlasov-Amp\`ere system \cite{article:Landau1946} in 1D1V:
\begin{gather}
\label{eqn:vlasovamp1}
f_{,t} + v f_{,x} - E(t,x) f_{,v} = 0, \quad E_{,t} = \mom_1(t,x) - J_0, \\ 
\label{eqn:vlasovamp2}
\mom_1(t,x) = \int_{-\infty}^{\infty} v f(t,x,v) \, dv, \quad
J_0 = \frac{1}{b-a} \int_{a}^{b} \mom_1(t=0,x) \, dx,
\end{gather}
where $t\in\reals_{\ge 0}$ is time, $x\in[a,b]$ is the spatial coordinate, $v \in \reals$ is the velocity coordinate,
$f(t,x,v):  \reals_{\ge 0} \times [a,b] \times \reals  \mapsto \reals_{\ge 0}$ is the particle density function (PDF) describing the distribution of electrons in phase space, $E(t,x):  \reals_{\ge 0} \times [a,b] \mapsto \reals$ is the electric field, $\mom_1(t,x):  \reals_{\ge 0} \times [a,b] \mapsto \reals$
is the electron momentum density, and $J_0 \in \reals$ is the stationary background ion momentum density.
The initial electric field can be computed from the Gauss' law:
\begin{equation}
E_{,x} = \rho_0 - \rho(t,x), \,
\rho(t,x) = \int_{-\infty}^{\infty} f(t,x,v) \, dv, \,
\rho_0 = \frac{1}{b-a} \int_{a}^{b} \rho(t=0,x) \, dx,
\end{equation}
where $\rho(t,x):  \reals_{\ge 0} \times [a,b] \mapsto \reals_{\ge 0}$
is the electron mass density and $\rho_0 \in \reals$ is the stationary background ion mass density.

A standard approach for numerically solving the Vlasov equation is to operator split the spatial and velocity derivative terms, allowing simpler time-stepping schemes to be applied to each piece separately (known as the Cheng-Knorr method
\cite{article:ChKn76}):
\begin{align}
\text{\bf Problem A:}& \quad f_{,t} + v f_{,x} = 0, \quad E_{,t} = \mom_1(t,x) - J_0; \\
\text{\bf Problem B:}& \quad f_{,t} - E f_{,v} = 0.
\end{align}
One disadvantage of such an operator split method is that the order of accuracy of fluid moments (e.g., mass, momentum, and energy) is restricted by the order of the operator splitting (second-order accuracy in the Cheng-Knorr \cite{article:ChKn76} case, which is based on Strang operator splitting \cite{article:St68}.)

\section{A Novel Operator Splitting Method}
\label{sec:closure}
We develop in this work a novel operator splitting approach for the Vlasov-Amp\`ere system \eqref{eqn:vlasovamp1}--\eqref{eqn:vlasovamp2}. We begin by discretizing the truncated velocity space, $v \in \left[V_{\text{min}},V_{\text{max}}\right]$, into non-overlapping {\it velocity bands} each of width and center:
\begin{equation}
\label{eqn:dv}
\dv = \left(V_{\text{max}} - V_{\text{min}}\right)/N_v, \quad
v_j = V_{\text{min}} + \left(j - \frac{1}{2}\right)\dv \hspace{1em} \text{for} \hspace{1em} j=1,\dots,N_v.
\end{equation}
where $V_{\text{max}}$ is the maximum of velocity domain, $V_{\text{min}}$ is the minimum velocity, and $N_v$ is the number of velocity layers. 
Within each velocity layer, we define the {\it local fluid moments} as,
\begin{equation}
\mom_{\ell}^{(j)} = \int_{v_{j-\frac{1}{2}}}^{v_{j+\frac{1}{2}}} v^{\ell} f(t,x,v)\,dv, \quad \ell=0,1,2,3,4,
\end{equation}
where $v_{j \pm \frac{1}{2}} = v_j \pm \Delta v/2$. Theoretically, we can take as many local fluid moments as desired; in the current work, we restrict ourselves to five (e.g., $\ell=0,1,2,3,4$). In practice, at least the first three moments should be included (mass, momentum, energy) to guarantee high-order accuracy on the lowest-order statistics. The more moments included, the more computationally complex the system in each velocity band will be. Five moments were chosen to balance accuracy and computational complexity.
We then take Vlasov equation \eqref{eqn:vlasovamp1}, multiply by $v^{\ell}$, and integrate over each
discrete velocity band, which, after integration by parts in velocity, becomes the following expression:
\begin{equation}
\mom_{\ell,t}^{(j)} + \mom_{\ell+1,x}^{(j)} - \left[Ev^\ell f\right]_{v_{j-\frac{1}{2}}}^{v_{j+\frac{1}{2}}}  = 
-\ell \mom_{\ell-1}^{(j)} E,
\end{equation}
for $j=1,\ldots,N_v$ and $\ell=0,1,2,3,4$. We note that the global moments (e.g., mass, momentum, and energy) can be directly computed from the local moments:
\begin{equation}
\label{eqn:global_moms}
    \mom_\ell = \sum_{j=1}^{N_v} \mom_\ell^{(j)}, \quad \text{for} \quad \ell=0,1,2,3,4.
\end{equation}
In each velocity band, we arrive at the following system of local moments:
\begin{equation}
\begin{split}
\label{eqn:bandeqn}
\vec{\mom}^{(j)}_{,t} + \mat{A}^{(j)} \, \vec{\mom}^{(j)}_{,x}
=-E \vec{S}^{(j)} 
+ E \vec{S}^{(j)}_{+}  f\bigl|_{v_{j+\frac{1}{2}}} -
E \vec{S}^{(j)}_{-}  f\bigl|_{v_{j-\frac{1}{2}}},
\end{split}
\end{equation}
where
\begin{equation}
\vec{\mom}^{(j)} = \begin{bmatrix}
\mom_0^{(j)} \\
\mom_1^{(j)} \\
\mom_2^{(j)} \\
\mom_3^{(j)} \\
\mom_4^{(j)}
\end{bmatrix}, \,
\mat{A}^{(j)} = 
\begin{bmatrix}
0 & 1 & 0 & 0 & 0 \\
0 & 0 & 1 & 0 & 0 \\
0 & 0 & 0 & 1 & 0 \\
0 & 0 & 0 & 0 & 1 \\
A_{51} & A_{52} & A_{53} & A_{54} & A_{55}
\end{bmatrix}, \, 
\vec{S}^{(j)}=
\begin{bmatrix}
0 \\
\mom_0^{(j)} \\
2\mom_1^{(j)} \\
3\mom_2^{(j)} \\
4\mom_3^{(j)} 
\end{bmatrix}, \,
\vec{S}^{(j)}_{\pm} = 
\begin{bmatrix} 1 \\
v_{j\pm\frac{1}{2}} \\
v_{j\pm\frac{1}{2}}^2 \\
v_{j\pm\frac{1}{2}}^3 \\
v_{j\pm\frac{1}{2}}^4
\end{bmatrix},
\end{equation}
where
\begin{equation}
\begin{split}
A_{51} = \frac{5\Delta v^4 v_j}{336}  - \frac{5\Delta v^2 v_j^3}{18} + v_j^5, \quad
A_{52} = \frac{-5\Delta v^4}{336} + \frac{5\Delta v^2 v_j^2}{6} - 5v_j^4, \\
A_{53} = \frac{-5 \Delta v^2 v_j }{6} + 10 v_j^3, \quad
A_{54} = \left(\frac{5}{18}\right)\left(\Delta v^2 - 36 v_j^2\right), \quad
A_{55} = 5 v_j.
\end{split}
\end{equation}
It can readily be shown that $\mat{A}^{(j)}$ is fully diagonalizable with only real eigenvalues.
For a full derivation, please see Vaughan \cite{article:ChristineVaughan2021}.

We note that the first three terms in equation \eqref{eqn:bandeqn} are entirely local to the $j^{\text{th}}$
velocity band, and only the final two terms in \eqref{eqn:bandeqn} are responsible for coupling the $j^{\text{th}}$
band to the $j+1$ and $j-1$ bands. This suggests a new operator-splitting approach:
\begin{align}
\label{eqn:prob_A}
\text{\bf Problem A:}& \quad \vec{\mom}^{(j)}_{,t} =  E \vec{S}^{(j)}_{+}  f\bigl|_{v_{j+\frac{1}{2}}} -
E \vec{S}^{(j)}_{-}  f\bigl|_{v_{j-\frac{1}{2}}}; \\
\label{eqn:prob_B}
\text{\bf Problem B:}& \quad \vec{\mom}^{(j)}_{,t} + \mat{A}^{(j)} \, \vec{\mom}^{(j)}_{,x}
=-E \vec{S}^{(j)}, \quad E_{,t} = \mom_1(t,x) - J_0.
\end{align}
Just as in Cheng and Knorr \cite{article:ChKn76}, we introduce Strang operator splitting \cite{article:St68}:
we will solve {\bf Problem A} on a half time step, $\frac{\Delta t}{2}$, followed by solving {\bf Problem B}
 on a full time step, $\Delta t$, and then finally again solving {\bf Problem A} on a half time.

\begin{proposition}
Under the assumption that the distribution function, $f$, is zero at $v=V_{\text{min}}$ and $v=V_{\text{max}}$, the operator splitting technique shown in \eqref{eqn:prob_A} and \eqref{eqn:prob_B} clearly introduces an
${\mathcal O}\left(\Delta t^2\right)$ splitting error into the accuracy of the PDF, $f$, but there is no operator
splitting error in the global moments: $\mom_0,\ldots,\mom_4$. 
\end{proposition}
\begin{proof}
The global moments, $\mom_0,\mom_1,\mom_2,\mom_3,\mom_4$, as defined by \eqref{eqn:global_moms} 
remain constant during {\bf Problem A}, since $\vec{S}_+^{(j)} = \vec{S}_-^{(j+1)}$ for all $j$,
and summing over $j$ in \eqref{eqn:prob_A} results in a telescoping
series such that the only remaining terms involve $f$ at $v=V_{\text{min}}$ and $v=V_{\text{max}}$, which
we assume are zero.
\end{proof}

\section{Numerical Method}
\label{sec:method}

We discretize \eqref{eqn:prob_A}--\eqref{eqn:prob_B} using the Lax-Wendroff discontinuous Galerkin (LxW-DG) as 
described in Johnson et al. \cite{article:JoRoVa2023}, which is based on previous versions of the LxW-DG work of Qiu et al. \cite{article:Qiu05} and Gassner et al. \cite{article:GasDumHinMun2011}. The method allows us to discretize in
space and time with piecewise polynomials of degree $\morder$ (i.e., ${\mathcal P}^{\morder}$), 
which in the case of smooth solutions results in errors of size ${\mathcal O}\left( \Delta t^{\morder+1} + 
\Delta x^{\morder+1}\right)$. We omit the details of this here and instead refer the reader to Johnson et al. \cite{article:JoRoVa2023}.

The novel aspect of this work is the introduction of {\bf Problem A} \eqref{eqn:prob_A}; and, as such,
we include details here on how this equation is solved numerically.
If $E\geq0$, we construct the following Taylor expansion comparing moments at $t=t^n$ and $t=t^n + \Delta t$
(see Vaughan \cite{article:ChristineVaughan2021}):
\begin{equation}
    \vec{\mom}^{(j)}(t^n+\dt) = \vec{\mom}^{(j)} - \dt\left(\vec{S}_+^{(j)}E\beta_+^{(j)} - \vec{S}_-^{(j)}E\beta_+^{(j-1)}\right) + \mathcal{O}(\dt^5),
\end{equation}
where
\begin{equation}
\begin{split}
    &\beta_+^{(j)} = \left(1- \frac{5\kappa}{2} + \frac{25 \kappa^2}{6} - \frac{125\kappa^3}{24}\right)\alpha_+^{(j)} 
    + \left(\frac{\kappa}{2} - \frac{5\kappa^2}{3} + \frac{25\kappa^3}{8}\right)\alpha_+^{(j-1)} \\
    &+ \left(\frac{\kappa^2}{6}-\frac{5\kappa^3}{8}\right)\alpha_+^{(j-2)} + \left(\frac{\kappa^3}{24}\right)\alpha_+^{(j-3)}, \, \kappa = \frac{5E\dt}{\dv}, \, 
    \alpha_+^{(j)} = \vec{N}_+^{(j)}\cdot\vec{\mom}^{(j)},
    \end{split}
\end{equation}
where $\vec{N}_+^{(j)}$ is an interpolated vector describing the distribution function at the upper interface of the velocity band $j$. There is a similar formula for $E \le 0$ \cite{article:ChristineVaughan2021}.


We define on each element, Gauss-Legendre quadrature points, $\xi_m$, for $m=1,2,\ldots,\morder$. We also define the Legendre polynomials $\psi_k(\xi)$ for $k=1,2,\ldots,\morder$. Putting all the pieces together, we find that
\begin{equation}
    \vec{\mom}_k^{(j)}(t^n+\dt) = \vec{\mom}_k^{(j)}(t^n) - \dt\left(\vec{\gamma}_k^{\left(j+\frac{1}{2}\right)} - \vec{\gamma}_k^{\left(j-\frac{1}{2}\right)}\right) + \mathcal{O}(\dt^5),
\end{equation}
where 
\begin{equation}
\vec{\gamma}_k^{\left(j\pm\frac{1}{2}\right)} = \vec{S}^{\left(j\right)}_{\pm} \left[\frac{1}{2}\sum_{m=1}^{\morder}\omega_m\psi_\ell(\xi_m)\beta_m^{\left(j\pm\frac{1}{2}\right)}\right],
\end{equation}
where
\begin{align}
    \beta_m^{\left(j-\frac{1}{2}\right)} &= \min(E(\xi_m),0) \, \beta_-^{(j-1)} + \max(E(\xi_m),0) \, \beta_+^{(j)},\\
    \beta_m^{\left(j+\frac{1}{2}\right)} &= \min(E(\xi_m),0) \, \beta_-^{(j)} + \max(E(\xi_m),0) \, \beta_+^{(j+1)}.
\end{align}

\section{Numerical Examples}
\label{sec:examples}

In this section we discuss the examples used to test the velocity-banded quadrature-based moment closure. We begin with a forced problem, using the method of manufactured solutions. We used this example to perform a convergence test to show the method converges up to fourth order in space and time. We also show that error in space and time as well as round-off error dominate the error over that produced by the discretization of the velocity space. Next, we show the results from the weak Landau damping problem, the strong Landau damping problem, and two-stream instability. Finally, we show the results from a plasma sheath example. 

\begin{table}
\begin{center}
\begin{tabular}{|c||c|c||c|c||c|c|}
\hline
{\normalsize $N$} & {\normalsize $\text{e}_N \left(M_{\text{O}}=1\right)$} & {\normalsize $\log_2\frac{\text{e}_{N/2}}{\text{e}_{N}}$}  & {\normalsize $\text{e}_N \left(M_{\text{O}}=2\right)$} & {\normalsize $\log_2\frac{\text{e}_{N/2}}{\text{e}_{N}}$}  & {\normalsize $\text{e}_N \left(M_{\text{O}}=3\right)$} & {\normalsize $\log_2\frac{\text{e}_{N/2}}{\text{e}_{N}}$} \\
\hline\hline
{\normalsize 20} & {\normalsize 5.203e-02} & -- & {\normalsize 2.995e-03} & -- & {\normalsize 1.144e-04} & -- \\\hline
{\normalsize 40} & {\normalsize 1.454e-02} & {\normalsize $1.840$} & {\normalsize 4.075e-04} & {\normalsize $2.878$} & {\normalsize 7.829e-06} & {\normalsize $3.869$} \\\hline
{\normalsize 80} & {\normalsize 3.890e-03} & {\normalsize $1.902$} & {\normalsize 4.964e-05} & {\normalsize $3.037$} & {\normalsize 4.927e-07} & {\normalsize $3.990$} \\\hline
{\normalsize 160} & {\normalsize 9.943e-04} & {\normalsize $1.968$} & {\normalsize 6.170e-06} & {\normalsize $3.008$} & {\normalsize 3.098e-08} & {\normalsize $3.991$} \\\hline
{\normalsize 320} & {\normalsize 2.502e-04} & {\normalsize $1.991$} & {\normalsize 7.715e-07} & {\normalsize $3.000$} & {\normalsize 2.335e-09} & {\normalsize $3.730$} \\\hline
\end{tabular} 
\caption{Relative $L^2$ errors for the forced manufactured problem with 200 bands for the Vlasov-Amp\`ere equation
with the ${\mathcal P}^1$ (Columns 2 and 3), ${\mathcal P}^2$ (Columns 4 and 5), and ${\mathcal P}^3$ (Columns 6 and 7) Lax-Wendroff DG schemes. Here $N$ refers to the number of spatial elements, $\Delta x = 2\pi/N$, and $\Delta t = \text{CFL} \times \Delta x/V_{\text{max}}$, where $V_{\text{max}} = 4$ and $\text{CFL} = 0.09$.
\label{table:manuf_sol_80}}
\end{center}
\end{table}

\subsection{Manufactured Solution}
\label{sub:conv_test}
Consider the forced Vlasov-Poisson equation:
\begin{equation}
\begin{split}
f_{,t} + vf_{,x} + Ef_{,v} = \psi(t,x,v), \quad
E_{,x} = \sqrt{\pi} - \int_{-\infty}^\infty f(t,x,v)\, dv,
\end{split}
\end{equation}
on the domain $(t,x,v)\in [0,0.1] \times [-\pi,\pi] \times [-4,4]$ with periodic boundary conditions in $x$. 
The forced exact solution is given by \cite{article:RoSe10}:
\begin{gather}
f(t,x,v) = \left(2-\cos\left(2x-2\pi t\right)\right)e^{-\frac{(4v-1)^2}{4}}, \,
E(t,x) = \frac{\sqrt{\pi}}{4}\sin\left(2x-2\pi t\right), \\
\begin{gathered}
\psi\left(t,x,v\right) = \frac{1}{2} S e^{-\frac{1}{4}\left(4v-1\right)^2}\bigg[\left(2\sqrt{\pi}+1\right)\left(4v - 2\sqrt{\pi}\right)
        - \sqrt{\pi}\left(4v-1\right) C \bigg],
\end{gathered}
\end{gather}
where $S:=\sin(2x-2\pi t)$ and $C:=\cos(2x-2\pi t)$. A convergence table for 3 different
polynomial orders are shown in Table \ref{table:manuf_sol_80}. In the convergence study, we
are decreasing $\Delta x$ and $\Delta t$ (both are proportional to $1/N$), but fixing $\Delta v$ (i.e., the number of velocity bands).

\begin{figure}
\begin{center}
\begin{tabular}{cc}
  (a)\includegraphics[width=.4\linewidth]{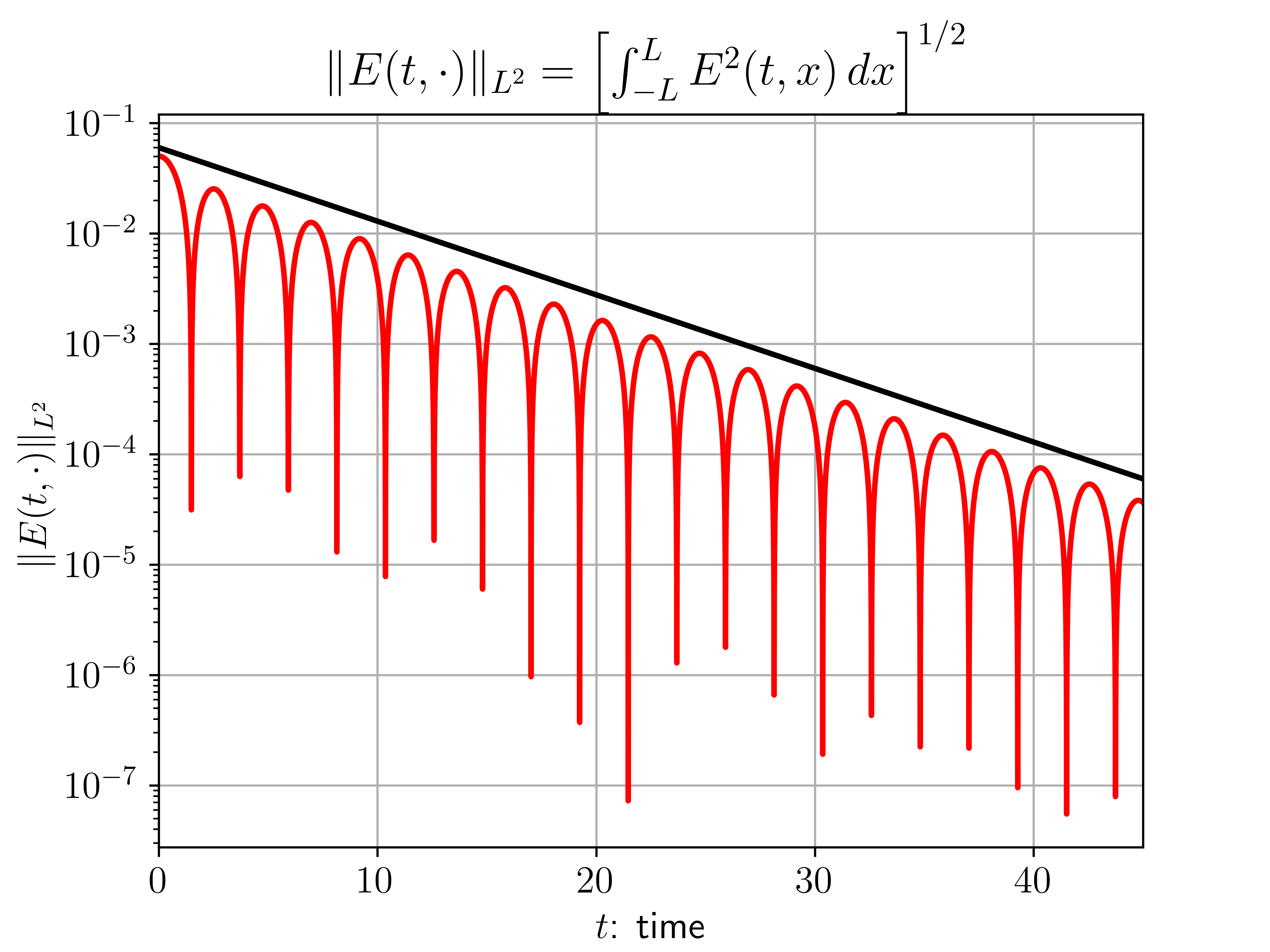} &
  (b)\includegraphics[width=.4\linewidth]{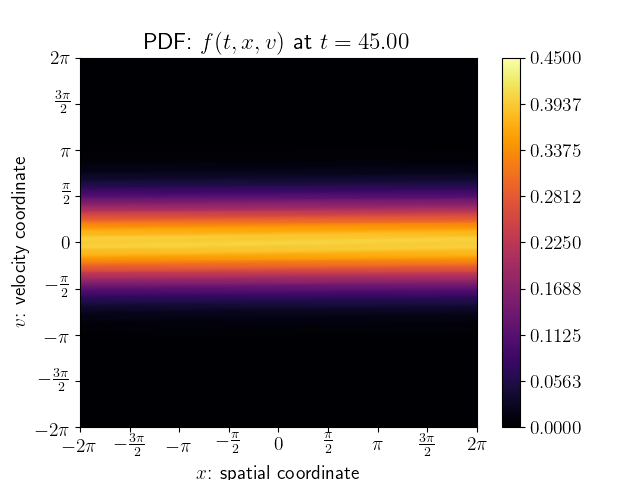}
\end{tabular}
\end{center}
\caption{Weak Landau damping problem. Panel (a) shows the exponential decay of the $L_2$-norm of the electric field and the line $0.06 e^{\gamma t}$ with the numerically computed decay rate $\gamma = -0.1536$,
and Panel (b) shows the reconstructed PDF.}
\label{fig:weak_landau}
\end{figure}

\subsection{Weak and Strong Landau Damping}
\label{sub:weak_landau}
Landau damping is a well-studied problem both numerically and analytically \cite{article:ChKn76,article:Heath10}, and the damping effect of the $L_2$ norm of the electric field is well-established and has been rigorously studied by Mouhot and Villani \cite{article:MoVi10}. This work earned Villani the 2010 Fields Medal. Given that the Landau damping problem is so well-studied, it is a favorite test example to study methods for plasma dynamics. The initial distribution function for the weak Landau damping problem is,
\begin{equation}
f(t=0,x,v) = \left(1 + \alpha\cos\left(kx\right)\right)\frac{1}{\sqrt{2\pi}}e^{-\frac{v^2}{2}},
\end{equation}
with $\alpha = 0.01$ (weak) or $\alpha = 0.5$ (strong) and $k = 0.5$ on the domain $(x,v)\in [-2\pi,2\pi] \times [-2\pi,2\pi]$.

The weak Landau simulation was run to $t=45$ with 40 cells in space and 80 bands in velocity. The simulation results are presented in Figure \ref{fig:weak_landau}. Panel (a) shows the damping of the $L_2$ norm of the electric field over time. Panel (b) shows the reconstructed distribution function at the final time. These results closely match those from the literature.

The strong Landau simulation was run to $t=60$ with 40 cells in space and 80 bands in velocity. The simulation results are presented in Figure \ref{fig:strong_landau}. Panel (a) shows the damping and growth of the $L_2$ norm of the electric field over time.  Panel (b) shows the reconstructed distribution function at the final time. These results closely match those from the literature.

\begin{figure}
\begin{center}
\begin{tabular}{cc}
  (a)\includegraphics[width=.4\linewidth]{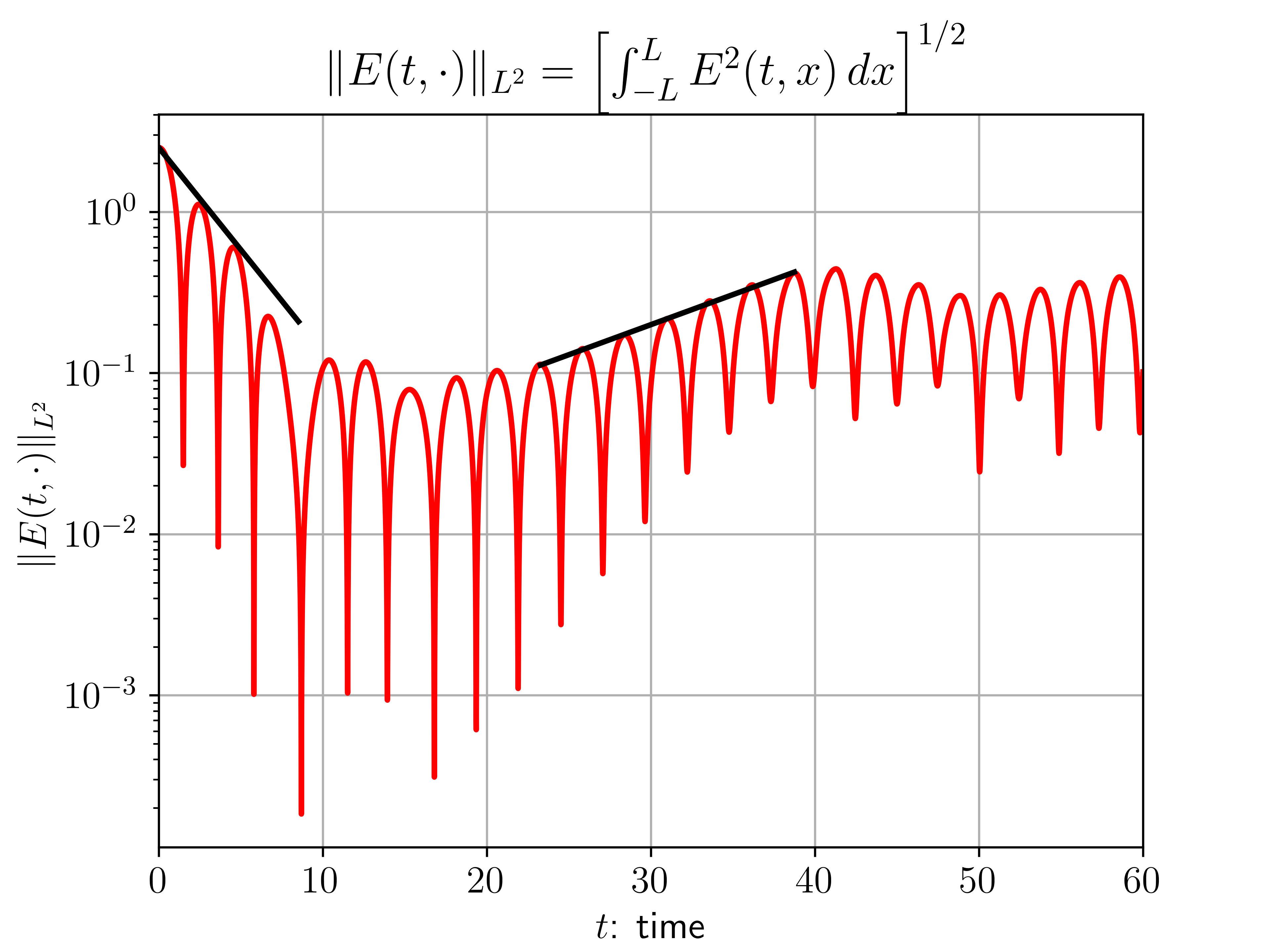} &
  (b)\includegraphics[width=.4\linewidth]{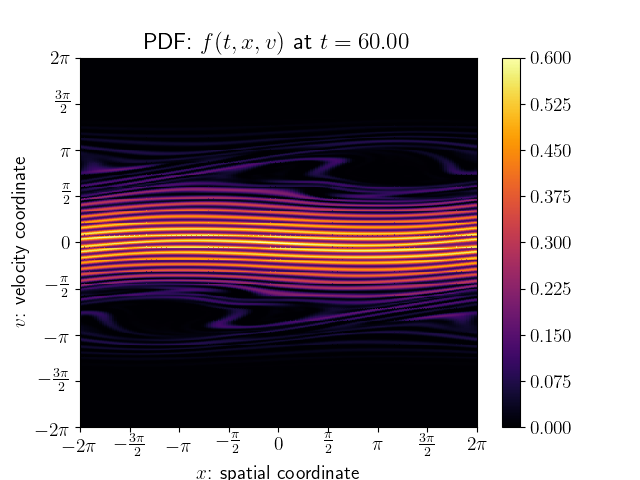}
\end{tabular}
\end{center}
\caption{Strong Landau damping problem. Panel (a) shows the initial exponential decay of the $L_2$-norm of the electric field ($\gamma_1 =  -0.2918$) followed by
exponential growth due to nonlinear effects ($\gamma_2 = 0.08584$),
and Panel (b) shows the reconstructed PDF.}
\label{fig:strong_landau}
\end{figure}

\subsection{Plasma Sheath}
\label{sub:plasma_sheath}
The plasma sheath problem is a classic example where the distribution function represents the electrons inside the region, $0\leq x\leq L$. We use a zero-inflow boundary condition, assuming no electrons can enter the domain. The initial condition is given by
\begin{equation}
f(0,x,v) = \frac{\rho_0}{\sqrt{2\pi\theta_0}} e^{-\frac{v^2}{2\theta_0}}.
\end{equation}
The problem is solved on the domain $(t,x,v)\in [0,140] \times [0, 1] \times [-0.2,0.2]$.
The setup shown here is identical to that in Seal \cite{thesis:seal}. 

Figure \ref{fig:plasma_sheath} displays the results of the simulation with 80 grid cells in space and
80 bands in velocity. Panel (a) shows the PDF and Panel (b) shows the density at $t=140$, where  
a plasma sheath has formed due to the net positive charge, keeping the electrons from leaving the domain. 
These results match those found in literature (e.g., Figure 20(b) of Seal \cite{thesis:seal}).

\begin{figure}
\begin{center}
\begin{tabular}{cc}
  (a)\includegraphics[width=.4\linewidth]{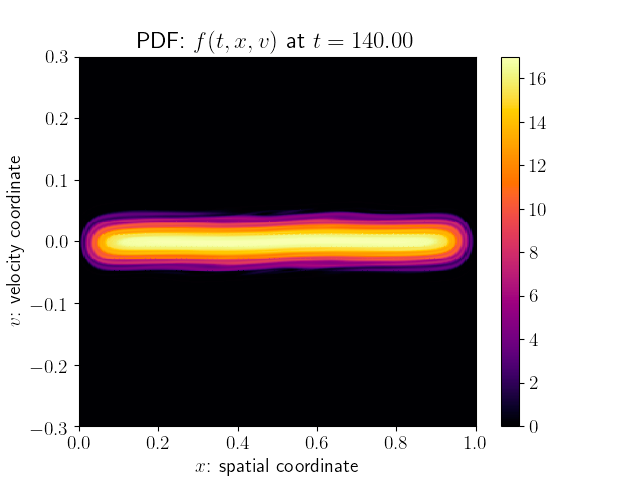} &
  (b)\includegraphics[width=.4\linewidth]{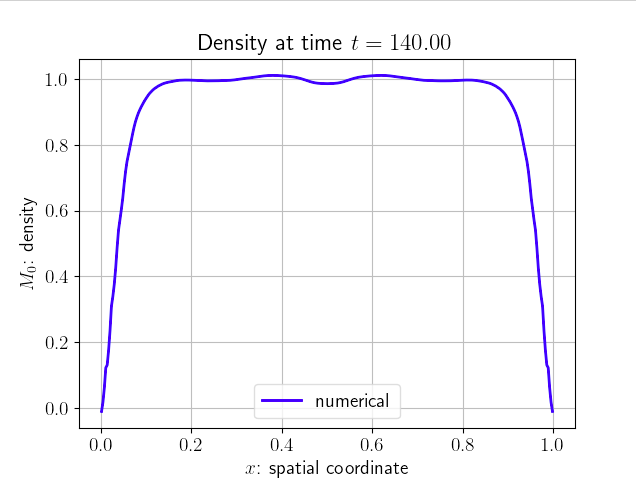}
\end{tabular}
\end{center}
\caption{Plasma sheath problem. Panel (a) shows the PDF, and Panel (b) shows the density.}
\label{fig:plasma_sheath}
\end{figure}

\section{Conclusions}
In this work, a novel discontinuous Galerkin scheme for the Vlasov-Amp\`ere system was developed. 
The scheme uses a spectral element moment-closure that results in a series of fluid models in each velocity
band. Strang operator splitting is introduced
to decouple these local fluid models and their inter-velocity coupling term. The coupling terms do not directly affect the global moments (e.g., mass, momentum, and energy); thus, high-order accuracy is achieved on the global moments.
 The scheme was tested on several numerical examples.
 
\section*{Acknowledgments}
We would like to thank the anonymous referees for their thoughtful comments and suggestions that helped to improve this paper. This research was partially funded by NSF Grants DMS--2012699 and DMS--2410538.

\bibliographystyle{siam}

\end{document}